\documentclass{amsart}
\usepackage{graphicx}
\usepackage[a4paper,scale=.75,hmarginratio=1:1]{geometry}

\usepackage{mathrsfs,hyperref}

\newcommand{\paren}[1]{\ensuremath{\left( #1\right) }}
\newenvironment{esn}{\begin{equation*}}{\end{equation*}}
\newcommand{\imf}[2]{\ensuremath{#1\!\paren{#2}}}
\newcommand{\cadlag}{c\`adl\`ag}
\newcommand{\fun}[3]{\ensuremath{#1:#2\to #3}}
\newcommand{\set}[1]{\ensuremath{\left\{ #1\right\} }}
\newcommand{\se}{\ensuremath{\bb{E}}}
\newcommand{\bb}[1]{\mathbb{#1}}
\newcommand{\fund}[3]{\ensuremath{#1:#2\mapsto #3}}
\newcommand{\na}{\ensuremath{\mathbb{N}}}
\newcommand{\abs}[1]{\,\left|#1\right|\,}
\newcommand{\eps}{\ensuremath{ \varepsilon}}
\newcommand{\p}{\ensuremath{ \sip  } }
\newcommand{\sip}{\bb{P}}
\newcommand{\indi}[1]{\si_{#1}}
\newcommand{\si}{{\ensuremath{\bf{1}}}}
\newcommand{\mc}[1]{\ensuremath{\mathscr{#1}}}
\DeclareMathOperator{\id}{Id} 
\newcommand{\proba}[1]{\ensuremath{\sip\! \left( #1 \right)}}
\newcommand{\defin}[1]{\textbf{#1}}
\theoremstyle{plain}
\newtheorem{teo}{Theorem}
\newtheorem{pro}{Proposition}

\newtheorem{cor}{Corollary}
\theoremstyle{definition}
\newtheorem*{defi}{Definition}
\newtheorem{example}{Example}

\newtheorem*{remark}{Remark}
\title[The Lamperti representation of CSBPs]{Proof(s) of the Lamperti representation of Continuous-State Branching Processes}
\author{Ma. Emilia Caballero}
\address[MEC]{Instituto de Matem\'aticas\\ Universidad Nacional Autonoma de M\'exico, \'Area de la investigaci\'on cient\'ifica, Circuito Exterior	Ciudad Universitaria, Coyoac\'an 04510, M\'exico, D.F. M\'exico}
\email{marie@matem.unam.mx}
\author{Amaury Lambert}
\address[AL]{UPMC Univ Paris 06, Laboratoire de Probabilit\'es et Mod\`eles Al\'eatoires, CNRS UMR 7599, Case courrier 188, 4 Place Jussieu, 75252 Paris Cedex 05, France.}
\email{amaury.lambert@upmc.fr}
\author{Ger{\'o}nimo Uribe Bravo}
\address[GUB]{Instituto de Investigaciones en Matem\'aticas Aplicadas y en Sistemas, Universidad Nacional Aut\'onoma de M\'exico, Mexico City, A.P. 20-726, Mexico}
\email[Corresponding author]{geronimo@sigma.iimas.unam.mx}
\date{\today}
\begin{document}
\begin{abstract}
This paper uses two new ingredients, namely stochastic differential equations satisfied by continuous-state branching processes (CSBPs), and a topology under which the Lamperti transformation is continuous, in order to provide self-contained proofs of Lamperti's 1967 representation of CSBPs in terms of spectrally positive L\'evy processes. The first proof is a direct probabilistic proof, and the second one uses approximations by discrete processes, for which the Lamperti representation is evident.
\end{abstract}
\maketitle
\section{Introduction}
\label{intro}
\subsection{The Lamperti representation theorem}
During the 1960s and early 70s, John Lamperti provided relationships  between \emph{L\'evy processes} and two other classes of Markov processes. The first class was that of \emph{continuous-state branching processes} (CSBPs for short) in \cite{csbp}, and the second one was that of \emph{positive self-similar} processes in \cite{lamp72} (then called \emph{positive semi-stable}; the reader might also wish to consult the recent survey \cite{surveyexponentialfunctionals}). Here, we are interested in the former, but both relationships have had a strong impact on recent research. From now on, we will refer to the first relationship as \emph{the} Lamperti representation. Roughly, it provides a one-to-one correspondence, via a simple \emph{random time change}, between CSBPs and (possibly killed) L\'evy process with no negative jumps. The Lamperti representation has proved useful in the study of CSBPs (as in \cite{binghamCBP}), but also in that of L\'evy processes (e.g. \cite{bertoinLocalTimesViaCSBP}) and superprocesses (see for example \cite{leGallSnakes}). 

Lamperti {announces} his representation theorem in \cite{csbp}, and assures that `proofs of the main results will appear elsewhere, but he never published them. Nine years have elapsed before this result was proved by I.S. Helland in \cite{weakConvUnderTimeChange} by discrete approximations. There is one missing step in Helland's paper, since non-conservative cases are not included. Also, M.L. Silverstein \cite[Theorem 4]{silversteinTimeChange} gives an `analytic paraphrase of Lamperti's result', namely, he proves by analytic methods, that CSBP laws are in one-to-one correspondence with Laplace exponents of L\'evy processes with no negative jumps. The Lamperti representation, as a path transformation, is not studied there.

Our goal here is to give two proofs of  the Lamperti representation. One is a direct proof of the Lamperti representation (including the absence of negative jumps) using probabilistic arguments (infinite divisibility, strong Markov property, martingales, stopping theorems, stochastic differential equations). The other one is a proof by discrete approximations, in the same vein as Helland, but using a new topology on Skorohod space.\\
\\
The state space we will work on is $E=[0,\infty]$ with any metric $\rho$ which makes it homeomorphic to $[0,1]$. We let $*$ stand for the convolution of measures, and use the convention $z+\infty=\infty$  for any $z\in E$.
\begin{defi}
\label{def : CSBP}
A \defin{continuous-state branching process}, in short \defin{CSBP}, is a conservative and c\`adl\`ag Markov process with values in $E$,  whose transition kernels $\paren{P_t}_{t\geq 0}$ satisfy the following \defin{branching property}:
\begin{esn}
\imf{P_t}{z_1,\cdot}*\imf{P_t}{z_2,\cdot}=\imf{P_t}{z_1+z_2,\cdot}
\end{esn}
for all $t\geq 0$ and $z_1,z_2\in E$. 
\end{defi}
\begin{remark}
We could also have defined CSBPs as stochastically continuous instead of \cadlag.  In the forthcoming Proposition \ref{prop : prelim}, we will see that 0 and $\infty$ are \emph{absorbing} states for a CSBP. We could give an analogous definition if the state-space were $[0,\infty)$ without the conservativity assumption; however, using $\infty$ as the cemetery point for the former we obtain an $E$-valued conservative process which will turn out to be a CSBP and have the Feller property with respect to the metric $\rho$.
\end{remark}
We now define the \emph{Lamperti transformation}, which acts on the Skorohod space of \cadlag\ trajectories with values in $E$ that are \emph{absorbed at zero and infinity}, that will be denoted $D$. More formally, $D$ consists of functions $\fun{f}{E}{E}$ which are \cadlag\ (so that in particular $\imf{f}{\infty-}:=\lim_{t\to\infty}\imf{f}{t}$ exists in $E$), such that $\imf{f}{\infty-}=\imf{f}{\infty}\in\set{0,\infty}$ and for which $\imf{f}{t}=0$ (resp. $=\infty$) implies that $\imf{f}{t+s}=0$ (resp. $=\infty$)  for all $s\geq 0$.

For any $f\in D$, first introduce the additive functional $\theta$ given by
\begin{esn}
\theta_t:=\int_0^t f(s)\, ds\in [0,\infty],
\end{esn}and let $\kappa$ denote the right-inverse of $\theta$ on $[0,\infty]$, given by
\begin{esn}
\kappa_t:=\inf\{u\ge 0: \theta_u >t\}\in [0,\infty]
\end{esn}using the convention $\inf\emptyset=\infty$. Define the \emph{Lamperti transformation}  $\fun{L}{D}{D}$ by\begin{esn}
\imf{L}{f}=f\circ\kappa
\end{esn}
where one remembers that $\imf{L}{f}(t)=f(\infty)$ if $\kappa_t=\infty$. 

Notice that  $0$ and $\infty$ indeed are also absorbing for $L(f)$. $L$ is a bijection of $D$. This can be checked by merely computing its inverse: setting $g=L(f)$ one rewrites $\kappa$ as
\begin{esn}
\kappa_t:=\int_0^t1/g(s)\, ds\in [0,\infty].
\end{esn}Then $f=g\circ \theta$, where $\theta$ is the right-inverse of $\kappa$.

It will always be implicit in what follows that a L\'evy process is a c\`adl\`ag process with independent and homogeneous increments, sent to $\infty$ at an independent exponential time, where it is understood that an exponential distribution with parameter zero means the distribution which assigns probability 1 to the value $\infty$. A \emph{spectrally positive} L\'evy process is a L\'evy process with no negative jumps. Recall (e.g. \cite{bertoinLevyP}) that the \emph{Laplace exponent} of a spectrally positive L\'evy process is a convex function $\Psi$ on $[0,\infty)$ satisfying
\begin{esn}
\imf{\se_x}{e^{-\lambda X_t}}=e^{-\lambda x +t\imf{\Psi}{\lambda}}\qquad t,x,\lambda\ge0.
\end{esn}
When $\Psi$ does not take positive values, $X$ is a.s. non-decreasing, and it is called a \emph{subordinator}.

If $\Psi(0)=0$, it is known that $X$ has infinite lifetime. If $q:=-\Psi(0)>0$, then it is easily seen that $X$ is the L\'evy process with Laplace exponent $q+\Psi$ killed at an independent exponential time with parameter $q$. Since in our setting, $X$ is set to $\infty$ after it is killed, we will consider that the killing time is the first jump with infinite size. This amounts to adding to the L\'evy measure a Dirac measure at $\{+\infty\}$ with mass $q$. \\
\\
Let us state the Lamperti representation theorem.

\begin{teo}[Lamperti representation of CSBPs \cite{csbp}]
\label{lampRep}
The Lamperti transformation is a bijection between continuous-state branching processes and L\'evy processes  with no negative jumps stopped whenever reaching zero. Specifically, for any CSBP $Z$, $L(Z)$ is a L\'evy process  with no negative jumps stopped whenever reaching zero; for any L\'evy process  with no negative jumps $X$ stopped whenever reaching zero, $L^{-1}(X)$ is a continuous-state branching process.
\end{teo}

There are two natural strategies for a proof of this theorem.\\
\\
 The first strategy is based on generators, and consists in proving a relationship of the type $A_{Z}f(x)=xA_{X}f(x)$, where $A_{Z}$ is the local generator of $Z$ and $A_{X}$ is that of $X$. Starting either with a CSBP $Z$ or with a L\'evy process $X$, one characterizes the Laplace transforms of one-dimensional distributions of the other process to obtain one part of the theorem. The other part can be obtained proving that the Lamperti transformation is onto (the Laplace exponent of any L\'evy process/CSBP is attained). This method was hinted at by M.L. Silverstein \cite[p.1045]{silversteinTimeChange} in the preparatory discussion of Theorem 4, which states that Laplace transforms of CSBP are in one-to-one correspondence with Laplace exponents of spectrally positive L\'evy processes (see the forthcoming Proposition \ref{LaplaceExponent}). However, this discussion was not meant to be a proof, but was rather meant to guess the form of the aforementioned correspondence (which is proved by purely analytical arguments).

We wish to provide a proof of the Lamperti representation theorem in this vein, that we call `direct probabilistic proof'. Our goal is to emphasize the probabilistic rationale for the appearance of a spectrally positive L\'evy process when applying the Lamperti transformation to CSBPs and of CSBPs when applying the inverse Lamperti transformation to spectrally positive L\'evy processes. In particular, we do not wish to use analytical arguments to prove surjectivity. The study of the Lamperti transformation uses martingales (as a substitute for the delicate use of generators), and the inverse is analyzed in the spirit of \cite{dawsonLi} where stochastic differential equations are shown to be satisfied by affine processes, which make explicit in the special case of CSBPs.\\
\\
The second strategy is based on discrete approximations. In the case of Markov branching processes with integer values (discrete-state branching processes, or DSBPs), the Lamperti representation in terms of time-continuous random walks (with no negative jumps `larger' than $-1$) is nearly evident (see below). After rescaling, this yields a one-to-one correspondence between rescaled DSBPs and certain compound Poisson processes which are in the domain of attraction of spectrally positive L\'evy processes. The second ingredient, due to \cite{lampertiberkeley}, is the fact that all CSBPs are limits of rescaled DSBPs. The third ingredient is a necessary and sufficient condition for a sequence $(Y_n)$ of DSBPs to converge to a certain (CSBP) process $Y$. Such a condition is given by I.S. Helland in \cite{weakConvUnderTimeChange} (see \cite{grimvall} for the case of a sequence of Bienaym\'e--Galton--Watson processes), and proved to be equivalent to the convergence (in finite-dimensional distributions \emph{and} weakly in the Skorohod topology) of the sequence $X_n:=L(Y_n)$ to a spectrally positive L\'evy process $X$. If the convergence of $(Y_n)$ is strong enough so as to guarantee continuity of the Lamperti transform, then Theorem \ref{lampRep} follows. This difficult step is carried out in \cite{weakConvUnderTimeChange}, even in the explosive case, where $\infty$ can be reached continuously in finite time by the CSBP (but in the exception of the non-conservative case, where  $\infty$  can be reached by a jump from a finite state). More specifically, the Lamperti transformation is not continuous w.r.t. the usual Skorohod topology due to possible explosive cases. If explosive cases are excluded, one can proceed as in \cite[Ch. VI,IX]{ethierKurtz} using properties of the time-change transformation studied by Whitt in \cite{usefulFunctionsWhitt}. If even explosive (but conservative) cases are included, one can follow the work of Helland, introducing the (non-metrizable) Stone topology on our Skorohod space $D$. The Lamperti transformation is still not continuous under this topology, but if $(W_n)$ converges weakly to $W$ w.r.t. this topology, then under certain conditions on $W$ and the sequence $(W_n)$, there is convergence of finite-dimensional distributions of $L(W_n)$ to those of $L(W)$. This proves sufficient to achieve the proof of Theorem \ref{lampRep}.

We will provide a proof of the Lamperti representation theorem in the same vein, that we call `proof through weak convergence', not completing the proof of Helland by allowing the non-conservative case, but rather, introducing a new topology on Skorohod space which will make the Lamperti transformation continuous on $D$. 

 

\subsection{Outline of the two proofs}
Section \ref{timeChange} is dedicated to the direct probabilistic proof of Theorem \ref{lampRep}, and Section \ref{weakConv} to its proof through weak convergence.

Beforehand, we will recall well-known properties of CSBPs and sometimes sketch their proofs. 



\subsubsection{Proof through martingales and stochastic calculus}
Let us outline  Section \ref{timeChange}. First, we prove that in continuous time and continuous state-space (both conditions are needed), a branching process cannot have negative jumps. 
Then we show that if $\fund{e_\lambda}{z}{\imf{\exp}{-\lambda z}}$ for any $\lambda>0$, then there exists a function $F$ (the negative of the branching mechanism of Proposition \ref{differentiability}) such that
\begin{esn}
M^{\lambda}_t:=\imf{e_\lambda}{Z_t}+\imf{F}{\lambda}\int_0^t Z_s\imf{e_\lambda}{Z_s}\, ds
\end{esn}is a martingale. Applying the optional stopping theorem to the time change $\kappa_t$, we get a differential equation satisfied by the Laplace transform of the marginal of the image $Y$ of $Z$ by the Lamperti transformation.
Solving this differential equation yields an expression which is very close to that known for a L\'evy process, when it is not stopped upon reaching 0. The conclusive step consists in proving that $Y$ indeed is a L\'evy process stopped upon reaching 0.

For the second part of the theorem, we can use the L\'evy--It\^o decomposition for the initial L\'evy process. We start with any spectrally positive L\'evy process $X$ with initial position $x>0$, L\'evy measure $\Lambda$ and Gaussian coefficient $\sigma$. Using the L\'evy--It\^o decomposition, we show that the image $Z$ of the process $X$ stopped upon reaching 0, by the inverse Lamperti transformation, satisfies 
\begin{align}
\label{eqn : SIE}
Z_t
=x&+a\int_0^t Z_s\, ds + \sigma \int_0^t\sqrt{Z_s} \,dB_s\nonumber
\\&+\int_0^t\int_0^{Z_{s-}}\int_{[1,\infty]}rN(ds, dv, dr)+
\int_0^t\int_0^{Z_{s-}}\int_{(0,1)}r\tilde{N}(ds, dv, dr),
\end{align}
for some real number $a$, where $B$ is a Brownian motion and $N$ is an independent Poisson measure ($\tilde N$ is the associated compensated measure) with intensity measure $ds\,dv\,\Lambda(dr)$. It is then immediate to deduce the branching property. This stochastic equation is similar to the ones shown to be satisfied by affine processes in \cite[(5.1)]{dawsonLi}, \cite{bertoinLeGallFlowsII} and \cite[(9)]{bertoinLeGallFlowsIII}. The Poisson integral in equation (\ref{eqn : SIE}) has the following interpretation: the role of the second coordinate of the Poisson measure $N$ is to \emph{mark} jumps in order to have them occur only if this mark is `below' the path of $Z$; thus, the jumps with size in $(r, r+dr)$ occur at a rate equal to $Z_{t}\, \Lambda(dr)$, that is, as in the discrete case discussed below, the branching process jumps at a rate which is \emph{linear} in the population size. 

\subsubsection{Proof through weak convergence}The second proof (Section \ref{weakConv}) relies on the approximation of spectrally positive L\'evy processes by compound Poisson processes and of CSBPs by (time-continuous) discrete state-space branching processes, abbreviated as DSBPs. 

\begin{defi}
A \defin{discrete space branching process} $Z=(Z_t;t\ge 0)$   is a c\`adl\`ag Markov process with values in $\overline\na=\na\cup\set{\infty}$ (sent to $\infty$ after possible blow-up), which jumps from state $i$ to state $i+j$, $j=-1,1,2,\ldots$, at rate $i \mu_{j+1}$,  where $(\mu_k)_{k\ge 0}$ is a finite measure on $\overline\na$ with zero mass at 1.
\end{defi}
The integer $Z_t$ can be interpreted as the size at time $t$ of a population where each individual is independently replaced at constant rate $\lambda:=\sum_k\mu_k$ by a random quantity  of individuals, equal to $k$ with probability $\mu_k/\lambda$.  As a consequence, it is easily seen that $Z$ satisfies the branching property.

To explain the heuristics behind the Lamperti transformation (implicit in \cite{csbp, silversteinTimeChange} and also found in \cite{bertoinMaphysto, kyprianouBook}), let us note that for any state $i\not\in\{0,\infty\}$, the size of the jump of $Z$ starting from $i$ does not depend on $i$. 
Thus, the jump chain of $Z$ is exactly that of the compound Poisson  process $X$ which goes from state $i$ to state $i+j$, $j=-1,1,2,\ldots$, at rate $ \mu_{j+1}$. The only difference between those two processes lies in the waiting times between two jumps. The Lamperti transformation is a random
time change that enables the paths of one process to be obtained from those of the other one by an appropriate modification of the waiting times. If $T_0=0$ and $T_1<T_2<\cdots$ are the successive jump times of $Z$, then the differences $\paren{T_i-T_{i-1}}_{i\geq 1}$ are conditionally independent given the successive states $\paren{Z_{T_i}}_{i\in\na}$ and conditionally on them, $T_i-T_{i-1}$ is exponential with parameter $\lambda Z_{T_{i-1}}$.
The important point is to notice that defining $Y$ as the Lamperti transform of $Z$ amounts to multiplying each waiting time  $T_i-T_{i-1}$ by $Z_{T_{i-1}}$; this turns the waiting time into an exponential variable with parameter $\lambda$, except when $Z_{T_{i-1}}=0$, since then the waiting time is infinite.
Therefore, $Y$ is equal to the compound Poisson process $X$ with rate $\lambda$ and jump distribution $\mu_{1+\cdot}$ stopped upon reaching zero. Of course, a similar sketch of proof can be achieved for the other direction of the Lamperti transformation.

It turns out that general CSBPs can be approximated by DSBPs at the level of finite-dimensional distributions if and only if the corresponding Lamperti transforms of the latter approximate spectrally positive L\'evy processes stopped whenever reaching zero (analogous to previous work of Grimvall presented in \cite{grimvall}). Therefore, one could hope to prove the Lamperti representation of CSBPs by weak convergence considerations. This program would be carried out in a very simple manner if the Lamperti transformation were continuous on Skorohod space but unfortunately this is not the case (Helland first reported such a phenomenon in \cite{weakConvUnderTimeChange}).  We will therefore have to circumvent this problem by using properties of our approximations which ensure weak convergence in Skorohod space with a modified topology which makes the Lamperti transformation continuous and implies convergence in the usual Skorohod space.


\subsection{Preliminary results}
\label{subsec : prelim results}
\begin{pro}
\label{prop : prelim}
 For a CSBP, both states $0$ and $\infty$ are absorbing, and for all 
$t,z\in(0,\infty)$,
\begin{esn}
\imf{P_t}{z,\set{0,\infty}}<1.
\end{esn} In addition, there is $\fun{u_t}{[0,\infty)}{[0,\infty)}$ such that
\begin{equation}
\label{silversteinDef}
\int_{[0,\infty]} e^{-\lambda z'}\, \imf{P_t}{z,dz'}=e^{-z\imf{u_t}{\lambda}}
\end{equation}
for $z\in [0,\infty]$ which satisfies the composition rule\begin{esn}
\imf{u_{t+s}}{\lambda}=\imf{u_t}{\imf{u_s}{\lambda}}. 
\end{esn}Finally, $\paren{P_t}_{t\geq 0}$ is a Feller semigroup.
\end{pro}
\begin{proof} The absorbing character of $0$ and $\infty$ is easily handled, the composition rule follows from the Markov property, while the Feller character can be dealt as in \cite[Lemma 2.2]{lampertiberkeley}. Indeed, the \cadlag\ character of the trajectories implies that $t\mapsto \imf{u_t}{\lambda}$ is continuous at zero, where it is equal to $\lambda$, and so the composition rule gives us continuity everywhere. The extended continuity theorem for Laplace transforms applied to \eqref{silversteinDef} implies that $P_tf$ is continuous whenever $f$ is (because the restriction of $f$ to $[0,\infty)$ would be continuous and bounded) and that it tends to $f$ pointwise as $t\to 0$. 
\end{proof}

We now provide further properties of $u_t$.


\begin{pro}
\label{differentiability}
For every $\lambda>0$, the function $t\mapsto \imf{u_t}{\lambda}$ is differentiable on $[0,\infty)$. Moreover,\begin{esn}
\frac{\partial\imf{u_{t}}{\lambda}}{\partial t}= F(\imf{u_t}{\lambda})\qquad t,\lambda\ge 0
\end{esn}where\begin{esn}
\imf{F}{\lambda}:=\left.\frac{\partial \imf{u_t}{\lambda}}{\partial t}\right|_{t=0} .
\end{esn}

The function $\Psi:=-F$ is called the \defin{branching mechanism} of $Z$.
\end{pro}
 This last result was proved in \cite{silversteinTimeChange}, using a delicate analytical proof, so we prefer to provide an elementary proof resting mainly on the composition rule. An only more slightly complicated argument found in \cite[Lemma 1, Chap V.2, p.413]{GikhmanSkorohod} enables one to generalize the above proposition to stochastically continuous multi-type continuous-state branching processes.
\begin{proof}
In this proof, we exclude the trivial case where $Z$ is a.s. constant, so that $\imf{u_t}{\lambda}\not= \lambda$ unless $t=0$.
First note that by a recursive application of the dominated convergence theorem, $\lambda\mapsto\imf{u_t}{\lambda}$ is infinitely differentiable in $(0,\infty)$ and strictly increasing.  Next observe that the Feller property of $Z$ gives the continuity of $t\mapsto \imf{u_t}{\lambda}$.


By the composition rule again,  we may write
\begin{equation}
\label{expInDomainForZ}
\imf{u_{t+h}}{\lambda}-\imf{u_t}{\lambda}
=\imf{u_{t}}{\imf{u_h}{\lambda}}-\imf{u_t}{\lambda}
=\imf{\frac{\partial \imf{u_t}{\lambda}}{\partial \lambda}}{\lambda'}\paren{\imf{u_h}{\lambda}-\lambda}.
\end{equation}for some $\lambda'\in[\lambda,\imf{u_h}{\lambda}]$.
Hence the increment $\imf{u_{t+h}}{\lambda}-\imf{u_t}{\lambda}$ has the same sign as $\imf{u_h}{\lambda}-\lambda$ and so, for equally-spaced partitions $\set{t_i}_{i}$ of $[0,t]$ with spacing $h$, we have:
\begin{esn}
\sum_{i}\abs{\imf{u_{t_{i+1}}}{\lambda}-\imf{u_{t_i}}{\lambda}}
=\mbox{sign}(\imf{u_h}{\lambda}-\lambda)\sum_{i}(\imf{u_{t_{i+1}}}{\lambda}-\imf{u_{t_i}}{\lambda})
=\abs{\imf{u_t}{\lambda}-\lambda}.
\end{esn}
We deduce that $t\mapsto \imf{u_t}{\lambda}$ has \emph{finite variation}  and hence, it is almost everywhere differentiable. Now thanks to \eqref{expInDomainForZ},  
\begin{esn}
\lim_{h\downarrow 0}
\frac{\imf{u_{t+h}}{\lambda}-\imf{u_t}{\lambda}}{\imf{u_{h}}{\lambda}-\lambda}= \frac{\partial \imf{u_t}{\lambda}}{\partial \lambda}
\end{esn}where the r.h.s. is nonzero, so choosing $t$ where $t\mapsto \imf{u_t}{\lambda}$ is differentiable, its right-derivative exists at $0$. This, along with the last display now yields the differentiability everywhere, as well as the following equality
\begin{equation}
 \label{eqn : plein de derivees partielles}
\frac{\partial\imf{u_{t}}{\lambda}}{\partial t}= \frac{\partial \imf{u_t}{\lambda}}{\partial \lambda}\cdot F(\lambda),
\end{equation}
where we have set
 \begin{esn}
\imf{F}{\lambda}:=\left.\frac{\partial \imf{u_t}{\lambda}}{\partial t}\right|_{t=0}.
\end{esn}Letting $h\downarrow 0$ in $\paren{u_h\circ u_t(\lambda)-u_t(\lambda)}/h$, we finally get
\begin{equation}
\label{eqn : magic}
\frac{\partial\imf{u_{t}}{\lambda}}{\partial t}= F(\imf{u_t}{\lambda}),
\end{equation}
which ends the proof.
\end{proof}

\begin{pro}
\label{LaplaceExponent}
The branching mechanism $\Psi$ is the Laplace exponent of a spectrally positive L\'evy process.
\end{pro}
This last proposition can  be found in \cite{silversteinTimeChange}, where  it is proved by analytical methods relying on completely monotone functions. Silverstein uses this proposition to prove uniqueness of solutions to the differential equation in Proposition \ref{differentiability}, which we only need in the proof by weak convergence and offer a simple argument for it. He additionally proves that any Laplace exponent of a killed spectrally positive L\'evy process can occur; we obtain this as a consequence of our approach to Theorem \ref{lampRep}. 
It is also proved in \cite{KawazuWatanabe}, where it is deduced mainly from It\^o's formula. We will rely on the convergence criteria for infinitely divisible probability measures as found in \cite[Thm. 15.14, p.295]{kallenberg}. 
\begin{proof}
Since for every $x\geq 0$, $\lambda\mapsto e^{-x\imf{u_t}{\lambda}}$ is the Laplace transform of a probability measure (on $[0,\infty]$) then $\lambda\mapsto \imf{u_t}{\lambda}$ is the Laplace exponent of a subordinator. Recalling that the Laplace exponent of a subordinator is minus its Laplace exponent as a spectrally positive L\'evy process (cf. \cite{bertoinLevyP}), it follows that for every $\eps>0$,\begin{equation}
\label{aproxEL}
\lambda\mapsto \paren{\lambda-\imf{u_\eps}{\lambda}}/\eps
\end{equation}is the Laplace exponent of a spectrally positive L\'evy process whose limit as $\eps\to 0+$, $\Psi$, is then the Laplace exponent of a spectrally positive L\'evy process.  Indeed, letting $G_{\eps}$ is the (infinitely divisible) law on $(-\infty,\infty]$ whose Laplace exponent  is \eqref{aproxEL}, the Helly-Bray theorem gives us a subsequence $\eps_k\to 0$ for which $G_{\eps_k}$ converges to an increasing \cadlag\ function $G$; we can interpret $G$ as the distribution of a probability measure $\mu$ on $[-\infty,\infty]$. To see that it doesn't charge $-\infty$, we use Fatou's lemma for convergence in law:\begin{esn}
\int e^{-\lambda x} \,\imf{G}{dx}\leq \liminf_{k\to\infty}\int e^{-\lambda x}\, \imf{G_{\eps_k}}{dx}\to e^{\imf{\Psi}{\lambda}}<\infty.
\end{esn}Actually, $\Psi$ is the log-Laplace transform of $G$: by the convergence in the preceding display we get\begin{esn}
\sup_k\int e^{-\lambda x} \imf{G_{\eps_k}}{dx}<\infty
\end{esn}for all $\lambda\geq 0$. Since for $\lambda'>\lambda\geq 0$, $e^{-\lambda' y}=\paren{e^{-\lambda y}}
^{\lambda'/\lambda}$, the $L_p$ criterion for uniform integrability implies that\begin{esn}
\int e^{-\lambda y}\,\imf{G_{\eps_k}}{dy}\to_{k\to\infty}\int e^{-\lambda y}\,\imf{G}{dy}
\end{esn}and so\begin{esn}
\int e^{-\lambda y}\,\imf{G}{dy}=e^{-\imf{F}{\lambda}}.
\end{esn}The same argument, when applied to $\lambda\mapsto \paren{\lambda-\imf{u_{t\eps}}{\lambda}}/\eps$, tells us that $t\Psi$ is the log-Laplace of a probability measure on $(-\infty,\infty]$, so that $G$ is infinitely divisible. The fact that its L\'evy measure does not charge $(-\infty,0)$ is deduced from \cite[Thm. 15.14, p.295]{kallenberg}.
\end{proof}

\section{Direct probabilistic proof}
\label{timeChange}

\subsection{The Lamperti transform of a CSBP}

Let $Z$ denote a CSBP and $\p_x$ its law when it starts at $x\in [0,\infty]$. First, we prove that $Z$ cannot have negative jumps. Fix $\delta>0$ and set
$$
J_\delta:=\inf\{t>0:Z_t-Z_{t-}<-\delta\}.
$$
Now let $n$ be any integer such that $x/n<\delta$ and let $(Z^{(i,j)};i\ge 1, j=1,\ldots,n)$ be independent copies of $Z$ whose starting point will be defined recursively on $i$. Also set $Z^{(i)}:=\sum_{j=1}^n Z^{(i,j)}$. Let $T^{(i,j)}_\delta$ denote the first hitting time of $(\delta, +\infty]$ by $Z^{(i,j)}$ and set $\tau_\delta^{(i)}:=\inf_{1\le j\le n} T^{(i,j)}_\delta$.
Now set the initial values of $Z^{(i,j)}$ as follows : $Z^{(1,j)}(0)=x/n<\delta$ for all $j$ and
$$
Z^{(i+1,j)}(0)=n^{-1}Z^{(i)}(\tau_\delta^{(i)})\qquad j=1,\ldots,n,\: i\ge 1,
$$
so that in particular $Z^{(1)}(0)=x$ and $Z^{(i+1)}(0)=Z^{(i)}(\tau_\delta^{(i)})$.
Next, define $I$ as
$$
I:=\min\{i\ge 1 : Z^{(i)}(\tau_\delta^{(i)}) >n\delta \}.
$$
Observe that by definition of $\tau^{(i)}_\delta$, all paths $(Z^{(i,j)}_t;t<\tau^{(i)}_\delta)$ remain below $\delta$, and so all paths $(Z^{(i)}_t;t<\tau^{(i)}_\delta)$ remain below $n\delta$.
Observe that each $Z^{(i)}$ has the same transition kernels as $Z$, and that $\tau_\delta^{(i)}$ is a stopping time for $(Z^{(i,j)}; j=1,\ldots,n)$, so that the concatenation, say $Z^\star$, in increasing order of $i=1,\ldots, I$, of the paths $Z^{(i)}$ all killed at $\tau^{(i)}_\delta>0$, has the same law as $Z$ killed at $T_{n\delta}$,
where
$$
T_{n\delta}:=\inf\{t\ge 0 : Z_t >n\delta \}.
$$
Now recall that for all $1\le i\le I$ and $1\le j\le n$, all paths $(Z^{(i,j)}_t;t<\tau^{(i)}_\delta)$ remain below $\delta$. Since these processes are CSBPs, they only take non-negative values, and therefore cannot have a negative jump of amplitude larger than $\delta$. Since CSBPs are Feller processes, they have no fixed time discontinuity and the independent copies $(Z^{(i,j)};j=1,\ldots,n)$ a.s. do not jump at the same time. As a consequence, $(Z^{(i)}_t;t<\tau^{(i)}_\delta)$ has no negative jump of amplitude larger than $\delta$. The same holds for $(Z^{(i)}_t;t\le\tau^{(i)}_\delta)$ because if $\tau^{(i)}_\delta$ is a jump time, it can only be the time of a positive jump. As a consequence, the process $Z^\star$ has no negative jump of amplitude larger than $\delta$, which implies
$$
T_{n\delta}<J_\delta.
$$
Letting $n\to\infty$ and because $\delta$ is arbitrarily small, this last inequality shows that $Z$ has no negative jumps.

Now define $Y$ as the image of $Z$ by the Lamperti transformation. Specifically, let $\kappa$ be the time-change defined as the  inverse of the additive functional $\theta:t\mapsto \int_0^t Z_s\, ds$ and let $Y$ be defined as $Z\circ\kappa$.  Recall Proposition \ref{prop : prelim} and the branching mechanism $-F$. 
We consider the process $M^\lambda$ defined as
\begin{esn}
M^{\lambda}_t=\imf{e_\lambda}{Z_t}+\imf{F}{\lambda}\int_0^t Z_s\imf{e_\lambda}{Z_s}\, ds.
\end{esn}
We now prove that $M^\lambda$ 
is a martingale under $\p$.
Thanks to \eqref{eqn : plein de derivees partielles},
\begin{esn}
\frac{\partial}{\partial t}\imf{\se_x}{\imf{e_\lambda}{Z_t}}=
-x \frac{\partial \imf{u_t}{\lambda}}{\partial t} e^{-x\imf{u_t}{\lambda}}=-x F(\lambda)\frac{\partial \imf{u_t}{\lambda}}{\partial \lambda} e^{-x\imf{u_t}{\lambda}}=F(\lambda)\frac{\partial}{\partial \lambda}\imf{\se_x}{\imf{e_\lambda}{Z_t}},
\end{esn}
which gives as a conclusion
\begin{esn}
\frac{\partial}{\partial t}\imf{\se_x}{\imf{e_\lambda}{Z_t}}=-\imf{F}{\lambda}\imf{\se_x}{Z_t\imf{e_\lambda}{Z_t}}.
\end{esn}
This last equality proves that $M^\lambda$ has constant expectation, and the fact that it is a  martingale follows from the Markov property of $Z$.

Now $\kappa_t$ is a stopping time, so we can use the optional stopping theorem to get that for any $s>0$,
\begin{esn}
\imf{\se_x}{M^\lambda_{\kappa_t\wedge s}}=\imf{e_\lambda}{x}
\end{esn}which translates into
\begin{esn}
\imf{\se_x}{\imf{e_\lambda}{Z_{\kappa_t\wedge s}}}=\imf{e_\lambda}{x}-\imf{F}{\lambda}\imf{\se_x}{\int_0^{\kappa_t\wedge s}Z_u\imf{e_\lambda}{Z_u}\, du}.
\end{esn}
By the dominated convergence theorem and the monotone convergence theorem applied respectively to the l.h.s. and r.h.s. as $s\to\infty$, one obtains
\begin{esn}
\imf{\se_x}{\imf{e_\lambda}{Z_{\kappa_t}}}=\imf{e_\lambda}{x}-\imf{F}{\lambda}\imf{\se_x}{\int_0^{\kappa_t}Z_u\imf{e_\lambda}{Z_u}\, du}
\end{esn}
so that by using the definition of $Y$ and the fact that
\begin{esn}
\int_0^{\kappa_t}Z_u\imf{e_\lambda}{Z_u}\, du
=\int_0^tZ_{\kappa_u}\imf{e_\lambda}{Z_{\kappa_u}}\, d\kappa_u
=\int_0^t\imf{e_\lambda}{Z_{\kappa_u}}\indi{Z_{\kappa_u>0}}\, du
\end{esn}
(since $Z_{\kappa_u}\,d\kappa_u=\indi{Z_{\kappa_u}>0}\, du$), we get the equality
\begin{equation}
\label{genY}
\imf{\se_x}{\imf{e_\lambda}{Y_t}}=\imf{e_\lambda}{x}-\imf{F}{\lambda}\int_0^t \imf{\se_x}{\imf{e_\lambda}{Y_s}\indi{Y_s>0}}\, ds.
\end{equation}We denote by $T_0$ the first hitting time of $0$ by $Y$. As a first consequence of \eqref{genY}, note that if we write $\imf{\se_x}{\imf{e_\lambda}{Y_t}\indi{Y_t>0}}=\imf{\se_x}{\imf{e_\lambda}{Y_t}}-\imf{\p_x}{T_0\leq t}$, the following differential equation is satisfied 
\begin{equation}
\label{diffEq}
\frac{\partial \imf{\se_x}{\imf{e_\lambda}{Y_t}}}{\partial t}+\imf{F}{\lambda}\imf{\se_x}{\imf{e_\lambda}{Y_t}}=F(\lambda)\imf{\p_x}{T_0\leq t}.
\end{equation} 
We can therefore use standard techniques of solving first order linear differential equations to deduce the following equality
\begin{equation}
\label{explicitExponent}
\imf{\se_x}{\imf{e_\lambda}{Y_t}}= e^{-\lambda x-\imf{F}{\lambda}t}+\imf{\se_x}{\paren{1-e^{-\imf{F}{\lambda}\paren{t-T_0}}}\indi{T_0\leq t}}.
\end{equation}

The last step is now to deduce that $Y$ is a L\'evy process stopped upon hitting 0.
In the case when $\imf{\p_x}{T_0=\infty}=1$ for some $x\in (0,\infty)$ the same property holds for all $x\in (0,\infty)$ and we conclude from \eqref{explicitExponent} that 
\begin{esn}
\imf{\se_x}{\imf{e_\lambda}{Y_t}}= e^{-\lambda x-\imf{F}{\lambda}t}.
\end{esn}
Then $Y$ is a L\'evy process which remains on $(0,\infty]$ when started there. It is therefore a subordinator and, from the last display, its Laplace exponent is $-F$. 

This step is more complicated when $\imf{\p_x}{T_0=\infty}<1$.  Because we would like to show how the L\'evy process emerges without appealing to analytical properties of the function $F$, we have been  able to achieve a proof  which makes no use of Proposition \ref{LaplaceExponent}. But since this proof is a bit long and technical, we propose hereafter a shorter one which uses Proposition \ref{LaplaceExponent}. Thanks to this proposition, there is  a spectrally positive L\'evy process $X$ with Laplace exponent $-F$, whose law we denote by $\mathbb{Q}$. 

We stick to the notation $T_0$ for both processes $X$ and $Y$. It is not difficult to arrive at the following equality
$$
\mathbb{Q}_x(e_\lambda(X_{t\wedge T_0}))=e^{-\lambda x-\imf{F}{\lambda}t}+\imf{\se_x}{\paren{1-e^{-\imf{F}{\lambda}\paren{t-T_0}}}\indi{T_0\leq t}}.
$$
Then thanks to \eqref{explicitExponent}, the only thing we have to check is that $T_0$ has the same law under $\p_x$ as under $\mathbb{Q}_x$. To see this, first recall that $T_0=\int_0^\infty Z_s ds$. Since the CSBP started at $x+y$ is the sum of two independent CSBPs started at $x$ and $y$ respectively, the distribution of $T_0$ under $\p_{x+y}$ is the convolution of the laws of $T_0$ under $\p_x$ and $\p_y$. We can therefore conclude that the distribution of $T_0$ under $\p_x$ is infinitely divisible on $[0,\infty]$, and that there is a nonnegative function  $\phi$ on $[0,\infty)$ such that $-\phi$ is the Laplace exponent of a subordinator and 
\begin{equation}
\label{eqn : T_0 et phi}
\imf{\se_x}{e^{-\lambda T_0}}=e^{-x\imf{\phi}{\lambda}}\qquad x,\lambda\ge 0.
\end{equation}
On the other hand, as is well-known \cite{bertoinLevyP},
$$
\imf{\mathbb{Q}_x}{e^{-\lambda T_0}}=e^{-x\imf{\varphi}{\lambda}}\qquad x,\lambda\ge 0,
$$
where $\varphi$ is the nonnegative function on $[0,\infty)$ characterised by $-F\circ\varphi =\mbox{Id}_{[0,\infty)}$. At this point, we have to make sure that $-F$ indeed takes positive values (i.e. $X$ is not a subordinator). On the contrary, if $F$ took only nonnegative values, then by \eqref{diffEq}, we would get
$$
\frac{\partial \imf{\se_x}{\imf{e_\lambda}{Y_t}}}{\partial t}=-\imf{F}{\lambda}\imf{\se_x}{\imf{e_\lambda}{Y_t}\indi{T_0> t}},
$$
so that all mappings $t\mapsto \se_x(e_\lambda(Y_t))$ would be nonincreasing. Letting $\lambda\to\infty$, we would get that the mapping $t\mapsto \p_x(Y_t=0)$ also is nonincreasing. But since $0$ is absorbing, this mapping is obviously nondecreasing, so that $\p_x(Y_t=0)=\p_x(Y_0=0)=0$ for all $t\ge 0$ and $x>0$. This contradicts the assumption that $Y$ hits 0 with positive probability.

If $(K_t^\lambda;t\ge 0)$ denotes the martingale obtained by taking conditional expectations of the terminal variable $\exp(-\lambda\int_0^\infty Z_s\, ds)$, we get\begin{esn}
K_t^\lambda = \imf{\exp}{ -\lambda\int_0^tZ_s\, ds-\phi(\lambda) Z_t}, 
\end{esn}so that in particular
\begin{esn}
e^{-x\phi(\lambda)}=\imf{\se_x}{\imf{\exp}{ -\lambda\int_0^tZ_s\, ds-\phi(\lambda) Z_t}}.
\end{esn}Informally, we evaluate the derivative w.r.t. $t$ of both sides at $t=0$ to obtain
\begin{equation}
\label{branchingMechanismFromIntegralOfZ}
0=-\lambda x e^{-x\phi(\lambda)}-xF(\phi(\lambda))e^{-x\phi(\lambda)},
\end{equation}so that $-F\circ \phi$ is the identity on $[0,\infty)$. This shows that $\phi=\varphi$, so that $T_0$ indeed has the same law under $\p_x$ as under $\mathbb{Q}_x$. It remains to give a formal proof of \eqref{branchingMechanismFromIntegralOfZ}. Write $K^\lambda$ as the product of the semimartingale $L_t^\lambda=\imf{\exp}{-\imf{\phi}{\lambda}Z_t}$ and the finite variation process $N_t^\lambda=\imf{\exp}{-\lambda\int_0^{t} Z_s\, ds}$; we can write $L_t^\lambda$ as\begin{esn}
L^\lambda_t=M_t-\imf{F}{\imf{\phi}{\lambda}}\int_0^t Z_se^{-\imf{\phi}{\lambda}Z_s}\, ds
\end{esn}where $M\equiv M^{\phi(\lambda)}$ is a (formerly defined) locally bounded martingale, in particular square integrable. Integration by parts gives us\begin{esn}
K_t^\lambda
=e^{-\imf{\phi}{\lambda}x}+\int_0^tN^\lambda_{s-}\, dM_s-\int_0^tN^\lambda_sL^\lambda_sZ_s\left[\imf{F}{\imf{\phi}{\lambda}}+\lambda \right]\, ds.
\end{esn}Since $N^\lambda$ is bounded, its stochastic integral with respect to $M^\lambda$ is a square integrable martingale. Taking expectations, the second summand vanishes, and since by stochastic continuity of $Z$, $t\mapsto\imf{\se_x}{L^\lambda_tN^\lambda_tZ_t}$ is continuous (and bounded), we get
\begin{esn}
0=\left.\frac{\partial}{\partial t}\imf{\se_x}{K^\lambda_t}\right|_{t=0}=-e^{-\imf{\phi}{\lambda}x}\left[\imf{F}{\imf{\phi}{\lambda}}+\lambda \right]
\end{esn}which implies \eqref{branchingMechanismFromIntegralOfZ}.

\subsection{The inverse Lamperti transform of a spectrally positive L\'evy process}
\label{Amaury'sConverse}
In this subsection, we consider a L\'evy process $X$ with no negative jumps, started at $x\ge0$, stopped at its first hitting time $T_0$ of 0, and possibly sent to $\infty$ after an independent exponential time. Using the well-known L\'evy-It\^o decomposition of $X$ \cite{bertoinLevyP, kyprianouBook}, we can write for every $t<T_0$
\begin{equation}
\label{eqn : LI}
X_t= x+at+\sigma B^X_t+P^X_t+M^X_t,
\end{equation}where $a$ is a real number, $\sigma$ is a nonnegative real number (the Gaussian coefficient), $B^X$ is a standard Brownian motion, $P^X$ is a compound Poisson process and $M^X$ is a square integrable martingale, all terms being independent and adapted to the same filtration. To be more specific about $P^X$ and $M^X$, we denote by $\Lambda$ the L\'evy measure of $X$, which is a $\sigma$-finite measure  on $(0,\infty]$ (see Introduction) such that $\int_{(0,\infty]}(1\wedge r^2)\Lambda(dr)<\infty$. Then there is a Poisson measure $N^X$ on $[0,\infty)\times(0,\infty]$ with intensity measure $dt\,\Lambda(dr)$, and associated compensated measure $\tilde{N}^X (dt,dr):=N^X(dt,dr)-dt\,\Lambda(dr)$ (defined for $r< 1$) such that
\begin{esn}
P^X_t:=\int_0^t\int_{[1,\infty]}r\,N^X(ds,dr)
\quad\text{ and }\quad
M^X_t:=\int_0^t\int_{(0,1)}r\,\tilde{N}^X(ds,dr),
\end{esn}where the second integral is the $L^2$ limit, as $\eps\to 0$, of
\begin{esn}
M^{X,\eps}_t:=\int_0^t\int_{(\eps,1)}r\,\tilde{N}^X(ds,dr).
\end{esn}
Notice that, at the first jump of $P^X$ of infinite size, $X$ jumps to $\infty$ and remains there. It will be implicit in the rest of the proof that equalities hold in $[0,\infty]$.

Recall that the inverse Lamperti transform $Z$ of $X$ is given as follows. Set 
\begin{esn}
\kappa_t:=\int_0^{t\wedge T_0}\frac{ds}{X_s} ,
\end{esn}and let $\theta$ be its inverse 
\begin{esn}
\theta_t:=\inf\{u\ge 0: \kappa_u >t\}\in [0,\infty],
\end{esn}so that $Z:=X\circ \theta$. To prove that $Z$ is a CSBP, we will use the following proposition.
\begin{pro}
\label{prop : SDE}
There is a standard Brownian motion $B^Z$, and an independent Poisson measure $N^Z$ on $[0,\infty)\times(0,\infty)\times(0,\infty]$ with intensity measure $dt\,dv\,\Lambda(dr)$ such that
\begin{align}
\label{eqn : SDE}
Z_t
=x&+a\int_0^t Z_s\, ds + \sigma \int_0^t\sqrt{Z_s} \,dB^Z_s\nonumber
\\&+\int_0^t\int_0^{Z_{s-}}\int_{[1,\infty]}rN^Z(ds, dv, dr)+
\int_0^t\int_0^{Z_{s-}}\int_{(0,1)}r\tilde{N}^Z(ds, dv, dr),
\end{align}
where $\tilde{N}^Z$ is the compensated Poisson measure associated with $N^Z$.
\end{pro}
\begin{proof}[Proof]

Define $\mc{G}$ as the time-changed filtration, that is, $\mc{G}_t=\mc{F}_{\theta_t}$. We denote by $(T_n,\Delta_n)_{n\ge 1}$ an arbitrary labelling of the pairs associating jump times and jump sizes of $Z$. By a standard enlarging procedure, we can assume we are also given an independent $\mc{G}$-Brownian motion $B$, an independent $\mc{G}$-Poisson point process $N$ on $[0,\infty)\times(0,\infty)\times(0,\infty]$ with intensity measure $dt\,dv\,\Lambda(dr)$, and an independent sequence $(U_n)_{n\ge 1}$ of random variables uniformly ditributed on $(0,1)$ such that $U_n$ is $\mc{G}_{T_n}$-measurable and independent of $\mc{G}_{T_n-}$.

As a first step, we define $B^Z$ and $N^Z$. Recall the L\'evy-It\^o decomposition \eqref{eqn : LI}. Notice that $Y:=B^X\circ\theta$ is a continuous local martingale w.r.t. $\mc{G}$, so we can define $B^Z$ as
\begin{esn}
B^Z_t:=\int_0^t \frac{\indi{Z_s\not=0}}{\sqrt{Z_s}}\, dY_s + \int_0^t \indi{Z_s=0}\, dB_s.
\end{esn}
Next, we define $N^Z$ as
\begin{esn}
N^Z(dt,dv,dr):=\sum_{n}\delta_{\{T_n,\,U_nZ_{T_n-},\, \Delta_n\}}(dt, dv, dr)+\indi{v>Z_{t-}} \,N(dt, dv, dr),
\end{esn}
where $\delta$ denotes Dirac measures.

The second step consists in proving that $B^Z$ is a $\mc{G}$-Brownian motion, and that $N^Z$ is an independent $\mc{G}$-Poisson point process with intensity $dt\,dv\,\Lambda(dr)$. 

Observe that $B^Z$ is a continuous local martingale w.r.t. $\mc{G}$, and that its quadratic variation in this filtration equals
\begin{esn}
<B^Z>_t = \int_0^t \frac{\indi{Z_s\not=0}}{Z_s}\, d\theta_s + \int_0^t \indi{Z_s=0}\, ds=\int_0^t (\indi{Z_s\not=0}+ \indi{Z_s=0})\, ds=t,
\end{esn}
because $d\theta_s=Z_s\, ds$. This shows that $B^Z$ is a $\mc{G}$-Brownian motion. For $N^Z$, let $H$ be a non-negative $\mc{G}$-predictable process, let $f$ be a two-variable non-negative Borel function, and let $R^X$ be the image of $N^X$ by the mapping $(t,r)\mapsto(\theta_t,r)$. Then by predictable projection,
\begin{eqnarray*}
\se\sum_nH_{T_n}\, f(U_n Z_{T_n-},\, \Delta_n)	&=& \se \int_0^1 du \int_{[0,\infty)}\int_{(0,\infty]}R^X(dt,dr)\,H_t\,f(uZ_{t-}, r)\\
	&=& \se \int_0^1 du \int_{[0,\infty)}d\theta_t\int_{(0,\infty]} \Lambda(dr)\,H_t\,f(uZ_{t}, r)\\
	&=&\se \int_0^\infty Z_{t}dt \int_{(0,\infty]}\Lambda(dr)\int_0^1 du\, H_t\,f(uZ_{t}, r) \\
	&=&\se \int_0^\infty dt \int_{(0,\infty]}\Lambda(dr)\int_0^{Z_{t}}dv\, H_t\,f( v, r).
\end{eqnarray*}
Now since 
\begin{align*}
&\se \int_{[0,\infty)}\int_{(0,\infty)}\int_{(0,\infty]}N(dt,dv,dr)\,\indi{v>Z_{t-}} \, H_t\,f(v,r)
\\&=\se \int_0^\infty dt \int_{(0,\infty]}\Lambda(dr)\int_{Z_{t}}^\infty dv\, H_t\,f( v, r),
\end{align*}
we deduce 
\begin{esn}
\se \int_{[0,\infty)}\int_{(0,\infty)}\int_{(0,\infty]}N^Z(dt,dv,dr) H_t\,f(v,r)=\se
 \int_0^\infty dt\int_{0}^\infty dv\int_{(0,\infty]}\Lambda(dr)\, H_t\,f( v, r),
\end{esn}
which shows that $N^Z$ is a $\mc{G}$-Poisson point process with the claimed intensity. Finally, since $B^Z$ is a $\mc{G}$-Brownian motion  and  $N^Z$ is a $\mc{G}$-Poisson point process, Theorem 6.3 on p.77 of \cite{ikedaWatanabeII} ensures that $B^Z$ and $N^Z$ are independent.

The last step is showing that $Z$ indeed solves \eqref{eqn : SDE}. 
We will refer to the successive terms in \eqref{eqn : SDE} as $A_t$ (Lebesgue integral), $\sigma W_t$ (Brownian integral), $U_t$ (Poisson integral), and $V_t$ (compensated Poisson integral). Since we want to prove that $X\circ\theta=x+A+\sigma W+U+V$, and since $a\theta_t=A_t$, it is enough to prove that 
 $B^X\circ\theta=: Y=W$, $P^X\circ\theta=U$, and $M^X\circ\theta=V$. 
 Denote by $T$ the absorption time of $Z$ at 0 and recall that by definition of $B^Z$,
\begin{esn}
\int_0^t \sqrt{Z_s}\,dB^Z_s = \int_0^t \sqrt{Z_s}\frac{\indi{Z_s\not=0}}{\sqrt{Z_s}}\,dY_s + \int_0^t \sqrt{Z_s}\indi{Z_s=0}\,dB_s,
\end{esn}where the second term vanishes. As a consequence, $W_t = Y_{t\wedge T}=Y_t$, which provides us with the first required equality. Since $P^X(\theta_t)$ is merely the sum of jumps of $X$ of size greater than 1 occurring before time $\theta_t$, it is also the sum of jumps of $Z$ of size greater than 1 occurring before time $t$. As a consequence,
\begin{esn}
P^X(\theta_t)= \sum_{n:T_n\le t} \Delta_n \indi{\Delta_n\ge 1}=\int_0^t\int_0^{Z_{s-}}\int_{(0,\infty]}r N^Z(ds,dv,dr)\,\indi{r\ge 1},
\end{esn}
which provides us with the second required equality. As for the third one, the same reasoning as previously yields the following, where limits are taken in $L^2$
\begin{align*}
M^X(\theta_t)
=\lim_{\varepsilon\downarrow 0}&\left(\sum_{n:T_n\le t} \Delta_n \indi{\varepsilon<\Delta_n< 1}- \theta_t \int_{(\varepsilon,1)}\Lambda (dr)\right)
\\= \lim_{\varepsilon\downarrow 0}&\left(\int_0^t\int_0^{Z_{s-}}\int_{(0,\infty]}r N^Z(ds,dv,dr)\,\indi{\varepsilon<r< 1}\right.
\\&- \left.\int_0^t ds \int_0^{Z_{s-}}dv\int_{(0,\infty]} \Lambda (dr)\,\indi{\varepsilon<r< 1}\right),
\end{align*}which indeeds shows that $M^X(\theta_t)=V_t$.\end{proof}

Now we want to prove that $X\circ\theta$ is a CSBP.
Thanks to Proposition \ref{prop : SDE}, we only need to check that any solution $Z$ to  \eqref{eqn : SDE} satisfies the  branching property. Let $Z^1$ and $Z^2$ be two independent copies of $Z$, one starting from $x_1$ and the other from $x_2$. Thanks to Proposition \ref{prop : SDE}, we can write the sum $\zeta$ of $Z_1$ and $Z_2$ as
\begin{esn}
\zeta_t:=Z^1_t+Z^2_t=x_1+x_2+A_t+\sigma W_t + U_t+V_t, 
\end{esn}
where, with obvious notation,
\begin{esn}
A_t:=a\int_0^t (Z^1_s+Z^2_s)\, ds,\quad\quad
W_t:=\int_0^t\sqrt{ Z^1_s} \,dB^1_s+\int_0^t\sqrt{ Z^2_s} \,dB^2_s,
\end{esn}
\begin{esn}
U_t:=\int_0^t\int_0^{Z^1_{s-}}\int_{[1,\infty]}rN^1(ds, dv, dr)+\int_0^t\int_0^{Z^2_{s-}}\int_{[1,\infty]}rN^2(ds, dv, dr),
\end{esn}
\begin{esn}
V_t:=\int_0^t\int_0^{Z^1_{s-}}\int_{(0,1)}r\tilde{N}^1(ds, dv, dr)+\int_0^t\int_0^{Z^2_{s-}}\int_{(0,1)}r\tilde{N}^2(ds, dv, dr),
\end{esn}
and $B^1, N^1, B^2, N^2$ are all independent and adapted to the same filtration, say $\mc{F} = (\mc{F}_t;t\ge 0)$. By a standard enlarging procedure, we can assume that we are also given an independent $\mc{F}$-Brownian motion $B$ and an independent $\mc{F}$-Poisson point process $N$ with intensity measure $dt\,dv\,\Lambda(dr)$. 

Notice that $W$ is a continuous local martingale with quadratic variation $t\mapsto \int_0^t \zeta_s ds$. Set
$$
B^\zeta_t := \int_0^t \frac{\indi{\zeta_s\not=0}}{\sqrt{\zeta_s}}\, dW_s + \int_0^t \indi{\zeta_s=0}\, dB_s.
$$
Then $B^\zeta$ is adapted to the filtration $\mc{F}$ and, letting $T$ denote the first hitting time of 0 by $\zeta$,
\begin{esn}
W_t=W_{t\wedge T}=\int_0^t \indi{\zeta_s\not=0}dW_s=\int_0^t\sqrt{ \zeta_s}\, dB^\zeta_s-\int_0^t\sqrt{ \zeta_s}\,\indi{\zeta_s=0}\, dB_s=\int_0^t\sqrt{ \zeta_s}\, dB^\zeta_s.
\end{esn}In addition, the quadratic variation of $B^\zeta$ in the filtration $\mc{F}$ is
$$
<B^\zeta>_t = \int_0^t \frac{\indi{\zeta_s\not=0}}{\zeta_s}\, d<W>_s + \int_0^t \indi{\zeta_s=0}\, ds = \int_0^t \indi{\zeta_s\not=0}\, ds + \int_0^t \indi{\zeta_s=0}\, ds = t,
$$
so that $B^\zeta$ is a $\mc{F}$-Brownian motion.
Now set
$$
N^\zeta (dt, dv, dr) = \indi{v< Z^1_{t-}} \,N^1(dt, dv, dr)+ \indi{Z^1_{t-}<v<\zeta_{t-}} \,N^2(dt, dv-Z^1_{t-}, dr)+\indi{v>\zeta_{t-}} \,N(dt, dv, dr)
$$
Then for any non-negative $\mc{F}$-predictable process $H=(H_t;t\ge0)$ and any two-variable non-negative Borel function $f$,
\begin{align*}
 &\int_{[0,\infty)}\int_{(0,\infty)}\int_{(0,\infty]}N^\zeta(dt,dv,dr)\,H_t \, f(v,r) \\
 &= \int_{[0,\infty)}\int_{(0,\infty)}\int_{(0,\infty]}N^1(dt,dv,dr)\,H_t\,\indi{v< Z^1_{t-}} \, f(v,r)\\
 &+ \int_{[0,\infty)}\int_{(0,\infty)}\int_{(0,\infty]}N^2(dt,dv,dr)\,H_t \,\indi{v< Z^2_{t-}} \,f(v+Z^1_{t-},r)\\
 &+ \int_{[0,\infty)}\int_{(0,\infty)}\int_{(0,\infty]}N(dt,dv,dr)\,H_t \,\indi{v> \zeta_{t-}} \, f(v,r),
\end{align*}
so that, taking $H_t=\indi{v<\zeta_{t-}}$ and $f(v,r)=r\indi{r\ge 1}$, we get
\begin{esn}
U_t=\int_0^t\int_{0}^{\zeta_{s-}}\int_{[1,\infty]}rN^\zeta(ds, dv, dr).
\end{esn}
In addition, by predictable projection,
\begin{esn}
\se
 \int_{[0,\infty)}\int_{(0,\infty)}\int_{(0,\infty]}N^\zeta(dt,dv,dr)\,H_t \, f(v,r) =\se \int_{[0,\infty)}\int_{(0,\infty)}\int_{(0,\infty]}dt\,dv\, \Lambda(dr)\,H_t\, f(v,r),
\end{esn}
so that $N^\zeta$ is a $\mc{F}$-Poisson point process with intensity $dt\,dv\, \Lambda(dr)$. Similarly, we could get that
\begin{esn}
V_t=\int_0^t\int_{0}^{\zeta_{s-}}\int_{(0,1)}r\tilde{N}^\zeta(ds, dv, dr), 
\end{esn}
concluding that
\begin{multline*}
\zeta_t:=x_1+x_2+a\int_0^t \zeta_s\, ds +\sigma\int_0^t\sqrt{ \zeta_s}\, dB^\zeta_s\\
+\int_0^t\int_0^{\zeta_{s-}}\int_{[1,\infty]}rN^\zeta(ds, dv, dr)+
\int_0^t\int_0^{\zeta_{s-}}\int_{(0,1)}r\tilde{N}^\zeta(ds, dv, dr).
\end{multline*}
Finally, since $B^\zeta$ is a $\mc{F}$-Brownian motion  and  $N^\zeta$ is a $\mc{F}$-Poisson point process, Theorem 6.3 of \cite{ikedaWatanabeII} ensures that $B^\zeta$ and $N^\zeta$ are independent. Pathwise uniqueness for \eqref{eqn : SDE}  is proved in \cite{dawsonLi} under the stronger integrability condition\begin{esn}
\int_{(0,\infty]}r\wedge r^2\,\imf{\Lambda}{dr}<\infty,
\end{esn}which excludes jumps of infinite size. We now sketch a proof, suggested by Zenghu Li, of pathwise uniqueness for lower semi-continuous solutions to \eqref{eqn : SDE}. As a consequence, we will conclude that $\zeta=Z^1+Z^2$ has the same law as the process $Z$ started at $x_1+x_2$, that is, $Z$ has the branching property. 

For each integer $n$, consider the equation\begin{align*}
Z_t
=x&+a\int_0^t Z_s\, ds + \sigma \int_0^t\sqrt{Z_s} \,dB^Z_s
\\&+\int_0^t\int_0^{Z_{s-}}\int_{[1,\infty]}r\wedge nN^Z(ds, dv, dr)+
\int_0^t\int_0^{Z_{s-}}\int_{(0,1)}r\tilde{N}^Z(ds, dv, dr),
\end{align*}Existence and pathwise uniqueness holds for this equation by Theorem 5.1 in \cite{dawsonLi}. Consider also two solutions  $Z'$ and $Z''$ to \eqref{eqn : SDE} and consider the first times $\tau'_n$ and $\tau''_n$ that they have a jump of magnitude greater than $n$. Set also $\tau_n=\tau'_n\wedge\tau''_n$. Then, $Z'$ and $Z''$ satisfy the above equation on $[0,\tau_n]$, and so they are indistinguishable on $[0,\tau_n]$. As $n\to\infty$, $\tau_n$ converges to the first instant when $Z'$ or $Z''$ have a jump of infinite size, say $\tau_{\infty}$, a jump that comes from  an atom of $N^Z$ of the form $\paren{\tau_\infty,v,\infty}$, so that both processes feature it. Since after this time both processes equal to $\infty$,  since the integral with respect to the Poisson process diverges, then $Z'$ and $Z''$ are indistinguishable.

\section{Proof through weak convergence}
\label{weakConv}
Here, we provide a second proof of Theorem \ref{lampRep}, this time through weak convergence. 
We use the fact that the Lamperti representation is easy to prove on discrete state-spaces, and introduce a topology on Skorohod space for which the inverse Lamperti transformation is \emph{continuous}. Then approximating L\'evy processes by compound Poisson processes, and CSBPs by discrete-state branching processes, we will deduce the Lamperti representation on the continuous state-space.

\subsection{Preliminaries}
Recall that $\rho$ is any metric on $E=[0,\infty]$ that makes $E$ homeomorphic to $[0,1]$. 
Recall the Skorohod-type space $D$ consisting of functions $\fun{f}{E}{E}$ which are \cadlag\ (so that in particular $\lim_{t\to\infty}\imf{f}{t}=\imf{f}{\infty}$), such that $\imf{f}{\infty}\in\set{0,\infty}$ and for which $\imf{f}{t}=0$ (resp. $=\infty$) implies that $\imf{f}{t+s}=0$  for all $s\geq 0$ (resp. $=\infty$).

For any $t\le \infty$, we denote by $\|\cdot\|_t$ the uniform norm on $[0,t]$, and  by $\rho_t^D$ the uniform distance with respect to $\rho$, that is,
$$
\rho_t^D(f,g):=\sup_{s\in[0,t]}\rho(f(s), g(s)).
$$
Let $\Lambda_t$ be the set of increasing homeomorphisms of $[0,t]$ into itself ($[0,\infty)$ if $t=\infty$), and define the metric $d_\infty$ on $D$ as
\begin{align*}
d_{\infty}(f,g):=1\wedge \inf_{\lambda\in\Lambda_\infty}\imf{\rho_\infty^D}{f,g\circ\lambda}\vee \|\lambda-\id\|_\infty.
\end{align*}

The proofs of the two following propositions can be found in Subsection \ref{technicalProofsWeakConvergence}.

\begin{pro}
\label{continuityInvLampTrans}
The inverse Lamperti transformation $L^{-1}$ is continuous on $\paren{D,d_\infty}$.
\end{pro}

\begin{remark}
The usual Skorohod topology on $[0,t]$ defined in \cite[Ch. 3.12]{billingsley} (resp. on $[0,\infty)$ defined in \cite[Ch. 3.16]{billingsley}) is induced by the metric $d_t$ (resp. $d$), where
\begin{align*}
\imf{d_t}{f,g}=\inf_{\lambda\in\Lambda_t}\, \imf{\rho_t^D}{f,g\circ\lambda}\vee \|\lambda-\id\|_t
\quad\left(\text{resp.}\quad
\imf{d}{f,g}=\int_0^\infty e^{- t}\imf{d_t}{f,g}\, dt\right).
\end{align*}Then $\imf{d}{f_n,f}\to 0$ as $n\to\infty$ if and only if for every continuity point $t$ of $f$, $\imf{d_t}{f_n,f}\to 0$ (cf. Lemma 1 in \cite[Ch. 3.16, p. 167]{billingsley}), which gives a precise meaning to saying that $d$ controls the convergence of $f_n$ to $f$ only on compact subsets $[0,t]$ of $[0,\infty)$. It is easy to see, and will be repeatedly used, that $d_{\infty}(f_n,f)\to 0$ as  $n\to\infty$ implies $\imf{d_t}{f_n,f}\to 0$ for every continuity point $t$ of $f$, so that convergence with $d_\infty$ implies convergence in the usual Skorohod space. We also point out that in general,\begin{esn}
\imf{d_\infty}{f,g} \leq \imf{\max}{\imf{d_t}{f,g} , \imf{d_\infty}{f\circ s_t, g\circ s_t}}
\end{esn}where $f\circ s_t:= \imf{f}{t+\cdot}$, since the right-hand side is obtained by taking the infimum over homeomorphisms which send $t$ to itself.
\end{remark}

We will also need the following technical result on stopped L\'evy processes, as well as its corollary.

\begin{pro}
\label{convOnSkoTypeSpace}
Let $X$ and $(X^n)_n$ be spectrally positive L\'evy processes with Laplace exponents $ \Psi, \Psi_n$ respectively.
If  for all $\lambda\geq 0$ we have\begin{esn}
\lim_{n\to\infty} \Psi_n(\lambda) = \Psi(\lambda),
\end{esn}
then $X^n$ stopped whenever reaching zero converges weakly in $(D, d_\infty)$ to $X$ stopped whenever reaching zero. The same result holds if the processes $\paren{X^n}_n$ are rescaled compound Poisson processes with jumps in $\set{-1}\cup\overline\na$.
\end{pro}

We will use the last proposition in the form of the following corollary (see e.g. Lemma 5.4 in \cite[p. 287]{limitseqbp}).
For any $a,b >0$, consider the scaling operator $S^a_b$ on Skorohod space which sends $f$ to $t\mapsto \imf{f}{a\cdot t}/b$.

\begin{cor}
\label{cor : practical use}
Let $X$ be a spectrally positive L\'evy process with Laplace exponent $\Psi$, started at $x\ge 0$ and stopped whenever reaching 0.
There are a sequence of integers $a_n\to\infty$, and a sequence $\paren{X^n}_{n}$ of compound Poisson processes started at $x_n\in \na$, stopped upon reaching 0, and whose jump distribution is concentrated on $\set{-1}\cup\overline \na$, such that the Laplace exponent $\Psi^n$ of $S_{n}^{a_n}(X_n)$ converges to $\Psi$ and the sequence $\paren{S_{n}^{a_n}(X_n)}_{n}$ converges weakly to $X$ in $(D, d_\infty)$.
\end{cor}

We begin the proof of Theorem \ref{lampRep} by studying the inverse Lamperti transformation.

\subsection{The inverse Lamperti transform of a spectrally negative L\'evy process}
Let $X$ and $\paren{X^n}_{n\in\na}$ be as in Corollary \ref{cor : practical use}. As we have noted in the introduction, denoting $L^{-1}$ the inverse Lamperti transformation, $L^{-1}(X^n)$ satisfies the branching property in $\overline\na$. Also it is obvious that $Z^n:=L^{-1}\circ S^{a_n}_n(X^n)$ satisfies the branching property in $n^{-1}\overline\na$ (e.g. because  $L^{-1}\circ S^a_b=S^{a /b}_b\circ L^{-1}$). 

Thanks to Proposition \ref{continuityInvLampTrans}, the sequence of branching processes $(Z^n)_n$ converges weakly in $(D,d_\infty)$ to the Markov process $Z:=L^{-1}(X)$ (time-changing a \cadlag\ strong Markov by the inverse of an additive functional gives another \cadlag\ strong Markov process, cf. \cite[Vol. 1, X.5]{dynkin}). To show that $Z$ is a CSBP, we have to check that it has inherited the branching property from the sequence $\paren{Z^n}_{n\in\na}$, and thanks to the Markov property, it is sufficient to check the branching property at any fixed time. The result is due to the  following two facts. First, because neither of the discrete branching processes $Z^n$ jumps at fixed times, neither does $Z$ jump at fixed times. Second, it is known that for any fixed time $t$, the mapping $D:f\mapsto f(t)$ is continuous at any $f$ which is continuous at $t$. As a conclusion, for any fixed time $t$, the mapping $D:f\mapsto f(t)$ is a.s. continuous at $Z$. This ends the proof.

\begin{remark}
Recall the usual topology on Skorohod space from the remark in the previous subsection. For this topology, the inverse Lamperti transformation $L^{-1}$ is \emph{not} continuous, and the problem is due to explosions as seen in the example below. 
\end{remark}

\begin{example}
\label{discontinuityLampTransExample}
 Consider an element $f$ of $D$ such that $\imf{f}{s}\to \infty$,\begin{esn}
\imf{\kappa_\infty}{f}=\int_0^\infty \frac{ds}{\imf{f}{s}}<\infty,
\end{esn}and note that its inverse Lamperti transform $\imf{L^{-1}}{f}$ blows up at $\kappa_\infty$. If we approximate $f$ by $f_n=f\indi{[0,n]}$, then the inverse Lamperti transform of $f_n$ is always zero after $\imf{\kappa_n}{f}$ so that it cannot converge to $\imf{L^{-1}}{f}$; it does converge to another limit however. This is illustrated in Figure \ref{discontinuousCharLampTrans}. An explanation of why the problem occurs is that $\imf{\kappa}{f}$ contracts $[0,\infty)$ into $[0,\kappa_\infty)$ and so to have convergence in Skorohod space of a sequence of functions when they approach a limit taking infinite values, we have to control the behaviour of the trajectories of the sequence on $[0,\infty)$ instead of only on its compacts subsets as with the usual metrics. 
\begin{figure}
\caption{Discontinuity of the inverse Lamperti transformation.}
\label{discontinuousCharLampTrans}
\includegraphics[width=\textwidth]{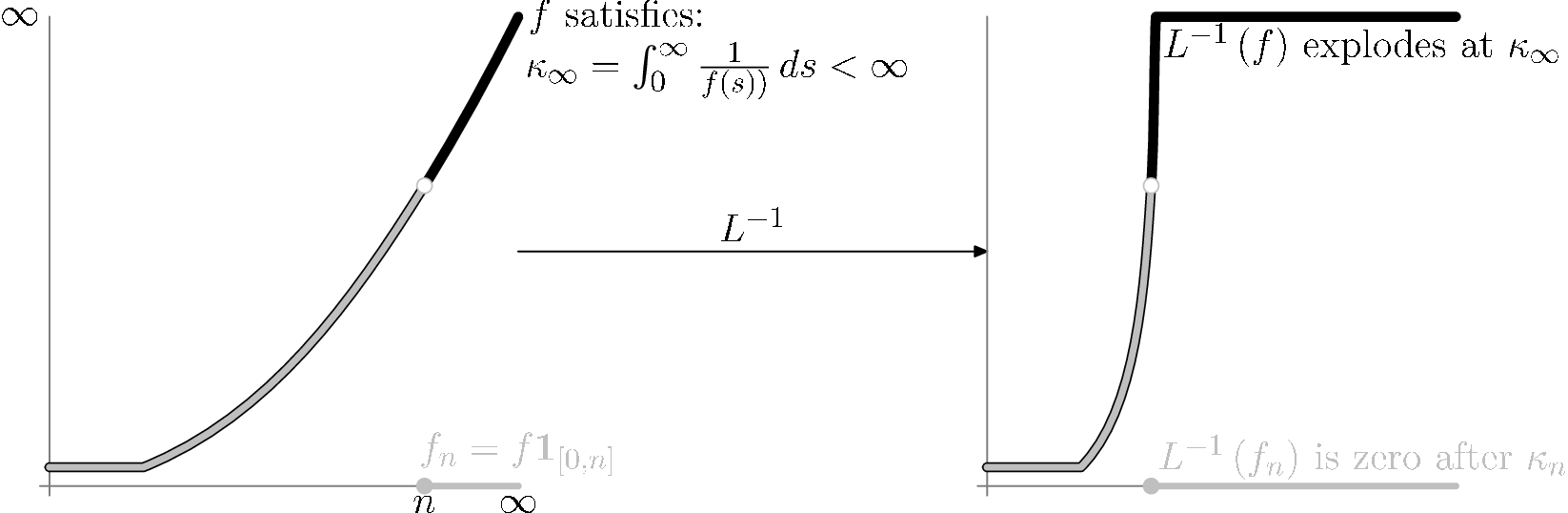}
\end{figure}
\end{example}

\subsection{The Lamperti transform of a CSBP}
\label{WeakConvergenceDirectSense}

Let $Z$ be a CSBP with law $\p_x$ when it starts at $x$. As we have shown in Propositions \ref{prop : prelim} and \ref{differentiability}, there are nonnegative real numbers $u_t(\lambda)$, $t,\lambda\ge 0$, such that 
\begin{esn}
\imf{\se_x}{e^{-\lambda Z_t}}=e^{-x\imf{u_t}{\lambda}},
\end{esn}
and $t\mapsto \imf{u_t}{\lambda}$ is differentiable on $[0,\infty)$. In addition, there is a real function $\Psi$ on $[0,\infty)$ called the branching mechanism of $Z$, such that 
\begin{equation}
\label{ode}
\frac{\partial \imf{u_t}{\lambda}}{\partial t}=-\imf{\Psi}{\imf{u_t}{{\lambda}}}\qquad t,\lambda\ge 0,
\end{equation}and $\Psi$ is the Laplace exponent of a spectrally positive L\'evy process.

Then let $X$ and $\paren{X^n}_{n\in\na}$ be as in Corollary \ref{cor : practical use}. Set $\tilde Z^n:=L^{-1}(S_{n}^{a_n}(X_n))$. 
%
%
%
%
As in the proof of the converse implication, each $\tilde Z^n$ satisfies the branching property in $n^{-1}\overline\na$, and thanks to Proposition \ref{continuityInvLampTrans}, the sequence of branching processes $(\tilde Z^n)_n$ converges weakly in $(D,d_\infty)$ to the Markov process $\tilde Z:=L^{-1}(X)$. (As remarked earlier, time-changing a \cadlag\ strong Markov by the inverse of an additive functional gives another \cadlag\ strong Markov process.)
We end the proof showing that the finite dimensional distributions of $\tilde Z^n$ converge to those of $Z$, which will entail the equality in distribution between $Z$ and $\tilde Z$, and subsequently between $L(Z)$ and $X$, since $X=L(\tilde Z)$.


Since  $\tilde Z^n$ is a branching process, there are real numbers $\imf{\tilde u^n_t}{\lambda}$ such that
\begin{esn}
\imf{\se_{x_n/n} 
}{e^{-\lambda \tilde Z^n_t}}= \imf{\exp}{-(x_n/n)\imf{\tilde u^n_t}{\lambda}}\qquad t,\lambda \ge 0,
\end{esn}
and we also have
\begin{equation}
\label{integralEquationDSBPs}
\imf{\tilde u^n_t}{\lambda}=\lambda-\int_0^t\imf{\Psi^n}{\imf{\tilde u^n_s}{\lambda}}\, ds .
\end{equation}By convergence of the sequence of branching processes $(\tilde Z^n)$, $\imf{\tilde u^n_t}{\lambda}$ converges pointwise to some nonnegative real number $\tilde u_t(\lambda)$ such that
\begin{esn}
\imf{\se_x 
}{e^{-\lambda \tilde Z_t}}= e^{-x\imf{\tilde u_t}{\lambda}}\qquad t,\lambda \ge 0.
\end{esn}Since $\Psi^n$ converges to $\Psi$ pointwise and they are convex on $(0,\infty)$, convergence is uniform on compact subsets of $(0,\infty)$; by taking limits in \eqref{integralEquationDSBPs}, we obtain
\begin{esn}
\imf{\tilde u_t}{\lambda}=\lambda-\int_0^t\imf{\Psi}{\imf{\tilde u_s}{\lambda}}\, ds.
\end{esn}Because of the local Lipschitz character of $\Psi$ on $(0,\infty)$ and Gronwall's lemma $\tilde u_t=u_t$. As a consequence,
\begin{esn}
\imf{\se_x 
}{e^{-\lambda \tilde Z_t}}= e^{-x\imf{ u_t}{\lambda}}=\imf{\se_x 
}{e^{-\lambda Z_t}}\qquad t,\lambda \ge 0,
\end{esn}so that $Z$ and $\tilde Z$ have the same law.
\subsection{Proof of propositions \ref{continuityInvLampTrans} and \ref{convOnSkoTypeSpace}}
\label{technicalProofsWeakConvergence}

\subsubsection{Proof of Proposition \ref{continuityInvLampTrans}}
There are two cases to consider since every element of $D$ either tends to $0$ or to $\infty$. At the outset however, there are some simple propositions that cover both. 


First of all, note that  if $\imf{d}{f_n,f}\to 0$ then $\imf{\kappa}{f_n}\to\imf{\kappa}{f}$ uniformly on compact sets of $[0,\imf{T_0}{f})$.

Second, note that if $\imf{d^c}{f_n,f}\to 0$ then $\imf{\theta}{f_n}\to\imf{\theta}{f}$ uniformly on compact sets of $[0,\imf{\kappa_{\imf{T_0}{f}}}{f})$. This follows from the following argument. It suffices to prove pointwise convergence on $[0,\imf{\kappa_{\imf{T_0}{f}}}{f})$; let $s<t<s'$ be three points on that interval. Then $\imf{\theta}{f}<\imf{T_0}{f}$ on $s,s'$ and $t$. By the preceding paragraph,\begin{esn}
\imf{\kappa_{\imf{\theta_s}{f}}}{f_n}\to \imf{\kappa_{\imf{\theta_s}{f}}}{f}=s<t
\end{esn}and so eventually, $\imf{\theta_s}{f}<\imf{\theta_t}{f_n}$. By the same argument, we see that eventually $\imf{\theta_t}{f_n}<\imf{\theta_{s'}}{f}$. By taking $s,s'\to t$, we see that $\imf{\theta_t}{f_n}\to\imf{\theta_t}{f}$. 

Note that the preceding two facts are true even if we are working with the metric $d$. The particular nature of the metric $d_\infty$ come into play at this stage: note that if $\imf{d_\infty}{f_n,f}\to 0$ then $\imf{f_n}{\infty}=\imf{f}{\infty}$ from a given index onwards.

We will now consider the case when $\imf{f}{\infty}=\infty$. Let $M>0$ be such that the $\rho$-diameter of $[M,\infty]$ is less than $\eps$. The quantity $\imf{L_{2M}}{f}=\sup\set{t\geq 0:\imf{f}{t}\leq M}$ is finite and $\inf_{s\geq \imf{L_{2M}}{f}}\imf{f}{s}\ge M$. Also, $\imf{\kappa_{\imf{L_{2M}}{f}}}{f}<\imf{\kappa_\infty}{f}$ (the rhs is $\imf{\kappa_{\imf{T_0}{f}}}{f}$) and so the preceding paragraphs tell us that $\imf{\theta}{f_n}\to\imf{\theta}{f}$  uniformly on $\imf{\kappa_{\imf{L_{2M}}{f}}}{f}$. Whitt's result on the continuity of time-changes \cite{usefulFunctionsWhitt} tells us that $\imf{L^{-1}}{f_n}\to\imf{L^{-1}}{f}$ (with respect to the Skorohod metric) on $[0,\imf{\kappa}{\imf{L_{2M}}{f}}{f}]$. Since\begin{esn}
\imf{\rho}{\imf{f}{s},\imf{f_n}{s}}<\eps
\end{esn}for $s>L_{2M}$, then $\imf{d_\infty}{\imf{L^{-1}}{f_n},\imf{L^{-1}}{f}}\to 0$. 

The remaining case, which is handled similarly, is when $\imf{f}{\infty}=0$. Suppose $\imf{f}{0}>0$, since otherwise there is nothing to prove. For $\eps>0$ small enough, we can introduce the (finite) quantity $\imf{L_\eps}{f}=\sup\set{t\geq 0:\imf{f}{t}>\eps}$. Since $L_\eps<T_0$, by the same arguments as above, we have that $\imf{L^{-1}}{f_n}\to\imf{L^{-1}}{f}$ (with respect to the Skorohod topology) on $[0,\imf{\kappa_{L_\eps}}{f}]$. Since, eventually, $f_n>2\eps$ on $[\imf{L}{\eps},\infty)$, then $\imf{d_\infty}{\imf{L^{-1}}{f_n},\imf{L^{-1}}{f}}\to 0$ as $n\to\infty$. 

\subsubsection{Proof of Proposition \ref{convOnSkoTypeSpace}}
Note that a given L\'evy process $X$ is either killed at an independent exponential time, or drifts to $\infty$, or to $-\infty$ or has $\liminf_{t\to\infty}X_t=-\infty$ and $\liminf_{t\to\infty}=\infty$ (it oscillates). When we stop a spectrally positive L\'evy process at 0 there are therefore three cases: either the stopped process jumps to $\infty$, or it drifts to $\infty$ without reaching 0 or it is stopped at 0 at a finite time. In any case,  the trajectories of the stopped process belong to $D$. The convergence of the Laplace exponents of the approximating sequence $X^n$ implies the convergence of the finite-dimensional distributions and so Skorohod's classical result implies that the convergence holds on $\paren{D,d}$ (cf. \cite[Thm. 15.17, p. 298]{kallenberg}). To study the convergence of the stopped processes on $\paren{D,d_\infty}$, we will use Skorohod's representation theorem to assume that, on a given probability space, $X^n$ converges almost-surely to $X$. Let $\imf{T_\eps}{X}$ denote $\inf\set{s\geq 0:X_t\leq \eps}\in [0,\infty]$; we will add the subscript $n$ when the stopping times are defined from $X^n$. Note that on the set $\imf{T_0}{X}<\infty$, $\imf{T_{0+}}{X}=\imf{T_0}{X}$, by the quasi-left-continuity of L\'evy processes (cf. \cite[Pro. I.2.7, p.21]{bertoinLevyP}). Stopping at the hitting time of zero is therefore a.s. continuous at $X$ (on $\paren{D,d}$, as can be seen in \cite{pages} and Lemma VI.2.10 in \cite[p. 340]{jacodShiryaevII}) and so $X^n$ stopped at zero, denoted $\tilde X^n$, converges almost surely to $\tilde{X}$ (equal to $X$ stopped at zero). We will now divide the proof in three cases.

\begin{description}
\item[$X$ drifts to $-\infty$ or oscillates] In this case, $T_0$ is finite almost surely. As we have remarked, $T^n_0\to T_0$ and so for $h>0$, $T^n_0\leq T_0+h$ from a given index onwards almost surely. Since L\'evy processes do not jump at fixed times, $X$ is continuous at $T_0+h$ for $h>0$ and so $\imf{d_{T_0+h}}{X,X^n}\to 0$ as $n\to\infty$. Note that\begin{esn}
\limsup_{n\to\infty}\imf{d_\infty}{\tilde X,\tilde X^n}\leq\limsup_{n\to\infty}\imf{d_{T_0+h}}{X,X^n}=0.
\end{esn}

\item[$X$ drifts to $\infty$]We will begin by verifying that the convergence of $X^n$ to $X$ (on $\paren{D,d}$) implies that we can uniformly control the overall infimum of the $X^n$. This is formally achieved in the following statement: given $\delta>0$ there exists some $M>0$ such that 
\begin{esn}
\p_{2M}(\inf_{s\geq 0} X_s <M) < \delta
\quad\text{and }\quad
\p_{2M}(\inf_{s\geq 0} X^n_s <M) < \delta
\end{esn}
from a given index onwards. For the proof, note that since $\Psi_n$ and $\Psi$ are strictly convex, we may  denote their largest roots by $\Phi_n$ and $\Phi$ respectively. When $X^n$ is a spectally positive L\'evy processes, the Laplace exponents $\Psi_n$ and  $\Psi$ restricted to $[\Phi_n,\infty)$ and $[\Phi,\infty)$ have inverses $\phi_n$ and $\phi$. The convergence of the Laplace exponents and their convexity
allow us to prove that
$\Phi_n \to \Phi$ as $n\to \infty$. When $X$ drifts to infinity, then $\Phi>0$ and from \cite[Thm. 1, p.189]{bertoinLevyP} and the above, we deduce 
\begin{esn}
\imf{\p_{2M}}{\inf_{s\geq 0} X_s <M},\limsup_{n\to\infty} \imf{\p_{2M}}{\inf_{s\geq 0} X^n_s <M}\leq e^{-M\Phi}.
\end{esn}By taking $M$ large enough, the claim follows. When the approximating sequence is constituted of rescaled left continuous compound Poisson processes, we adapt the proof of \cite[Thm. 1, p.189]{bertoinLevyP} to arrive at the same conclusion. 

Since $X$ drifts to $\infty$, it reaches arbitrarily high levels, and since $X^n$ converges to $X$ on  $\paren{D,d}$, then $X^n$ will also reach arbitrarily high levels. Coupled with our control on the infimum, we will see that from a given (random) time onwards and from a given index, $X^n$ and $X$ are close since they are above a high enough barrier. Formally, we will now prove that $\tilde X^n$ converges to $\tilde X$ in probability (using $d_\infty$): for any $\eps,\delta>0$ let $M>0$ be such that the $\rho$-diameter of $[M, \infty]$ is less than $\eps$ and $\imf{\exp}{-M\Phi}<\delta/2$.
We introduce the stopping time:
\begin{align*}
S_{3M}=\inf\set{s\geq 0 : X_s >3M},
\end{align*}
as well as the corresponding hitting times times $S^n_{2M}$ of $[2M,\infty)$ for $X_n, n=1,2,\cdots $. Observe that
\begin{align*}
&\imf{\p}{\imf{d_{\infty}}{\tilde X^n, \tilde X}> \eps}
\\&= \imf{\p}{\imf{d_{\infty}}{\tilde X^n, \tilde X}> \eps,T_0<\infty}+\imf{\p}{\imf{d_{\infty}}{\tilde X^n, \tilde X}>\eps,T_0=\infty}.
\end{align*}
The first summand of the right-hand side of the preceding inequality converges to zero by the arguments of the previous case. Consider $h>0$ and let us bound the second summand by\begin{esn}
\imf{\p}{C_n}+\imf{\p}{D_n}
\end{esn}where\begin{esn}
C_n = \set{\imf{d_{S_{3M}+h}}{X^n, X}> \eps,T_0=\infty}
\end{esn}and\begin{esn}
D_n = \set{\imf{d_{\infty}}{\paren{X^n_{S_{3M}+h+t}}_{t\geq 0}, \paren{X_{S_{3M}+h+t}}_{t\geq 0}}> \eps,T_0=\infty}.
\end{esn}Since L\'evy processes do not jump at fixed times, the Strong Markov property implies that almost surely $X$ does not jump at time $S_{3M}+h$ so that $\imf{d_{S_{3M}+h}}{X^n,X}\to 0$ almost surely. Hence\begin{esn}
\lim_{n\to \infty} \p(C_n) = 0.
\end{esn}This also implies that from a given index onwards, $S^n_{2M}\leq S_{3M}+h$ so that\begin{esn}
\proba{S^n_{2M}> S_{3M}+h}\to 0.
\end{esn}Hence, it remains to bound $\proba{D_n,S^n_{2M}\leq S_{3M}+h}$. If the $d_\infty$ distance between $\paren{X_{S_{3M+h+t}}}_{t\geq 0}$ and $\paren{X^n_{S_{2M+h+t}}}_{t\geq 0}$ is to be greater than $\eps$ while $S^n_{2M}\leq S_{3M+h}$ then either $X^n$ goes below $M$ after $S^n_{2M}$ or $X$ goes below $M$ after $S_{3M}$. The probability of both events is smaller than $\delta/2$ from a given index onwards because of our choice of $M$, so that\begin{esn}
\limsup_n\imf{\p}{\imf{d_{\infty}}{\tilde X^n, \tilde X}> \eps}\leq \delta
\end{esn}for every $\delta>0$. We conclude that  $\imf{d_\infty}{X^n,X}\to 0$ in probability.

\item[$X$ jumps to $\infty$] This case is characterized by $q:=-\imf{\Psi}{0}>0$. It can be reduced to the $q=0$ case by means of an independent exponential variable of rate $q$: if $X'$ is a L\'evy process whose Laplace exponent is $\Psi-\imf{\Psi}{0}$ and $T$ is an exponential variable with mean $1$ independent of $X'$ and we define $X''$ as $X'$ sent to $\infty$ at time $T/q$, then $X''$ has the same law as $X$. If $X'^n$ is a L\'evy process with Laplace exponent $\Psi_n-\imf{\Psi_n}{0}$ (and $q_n:=-\imf{\Psi_n}{0}$) then $X'^n$ converges in law to $X'$ on $\paren{D,d}$; as before, we will use Skorohod's representation theorem so that convergence holds almost surely on a given probability space. We now extend that space so as to have an additional mean $1$ exponential variable $T$ independent of $X'$ and $\paren{X'^n}_{n\in\na}$ and define on that space $X''$ and $X''^n$ as above by killing $X'$ and $X'^n$ at times $T/q$ and $T/q_n$ respectively. Since $q_n\to q$ by hypothesis,  and $X'$ is continuous at time $T/q+h$ almost surely (for any $h>0$), then $\imf{d_{T/q+h}}{X'^n,X'}\to 0$ and since $T/q_n\leq  T/q+h$ from a given index onwards, then $\imf{d_\infty}{X''^n,X''}\to 0$ almost surely. 
\end{description}

\section*{Acknowledgements}
The authors wish to thank Zenghu Li and Steve Evans for their comments on a previous version of this paper, and Sylvie M\'el\'eard, as well as Zenghu Li again, for their help with stochastic integral equations. G.U.B.'s research was supported by CoNaCyT grant No. 174498. Logistic and financial support received from PROYECTO
PAPIITT-IN120605 and Instituto de Matematicas. A.L. is very thankful to the staff and colleagues at the Instituto for their hospitality during his stay.

\bibliography{treescsbp}
\bibliographystyle{acmtrans-ims}

\end{document}